\documentclass{article}
\usepackage[final]{graphicx}
\usepackage{psfrag}
\usepackage{amsmath,amsfonts,amssymb,amsxtra,subeqnarray}

\newcommand {\Real}{\ensuremath{{\mathbb{R}}}}
\newcommand {\Natural}{\ensuremath{{\mathbb{N}}}}

\newcommand{\V}{\ensuremath{\mathcal V}}

\newcommand{\setH}{\ensuremath{\mathcal H}}
\newcommand{\M}{\ensuremath{\mathcal M}}

\newcommand{\setE}{\ensuremath{\mathcal E}}

\newcommand{\Y}{\ensuremath{\mathcal Y}}
\newcommand{\G}{\ensuremath{\mathcal G}}

\newcommand{\one}{\ensuremath{{\mathbf{1}}}}

\newtheorem{theorem}{Theorem}
\newtheorem{algorithm}{Algorithm}
\newtheorem{corollary}{Corollary}

\newtheorem{lemma}{Lemma}

\newtheorem{remark}{Remark}
\newtheorem{assumption}{Assumption}
\newtheorem{example}{Example}

\newenvironment{proof}{\noindent {\bf Proof.}}{\hfill \hspace*{1pt}\hfill$\blacksquare$}

\begin{document}

\title{Synchronization under matrix-weighted Laplacian}
\author{S. Emre Tuna\footnote{The author is with Department of
Electrical and Electronics Engineering, Middle East Technical
University, 06800 Ankara, Turkey. Email: {\tt
tuna@eee.metu.edu.tr}}} \maketitle

\begin{abstract}
Synchronization in a group of linear time-invariant systems is
studied where the coupling between each pair of systems is
characterized by a different output matrix. Simple methods are
proposed to generate a (separate) linear coupling gain for each
pair of systems, which ensures that all the solutions converge to
a common trajectory. Both continuous-time and discrete-time cases
are considered.
\end{abstract}

\section{Introduction}

Synchronization (consensus) of linear systems with general
dynamics (as opposed to first- or second-order integrators) has
been thoroughly investigated in the last decade. Early results
established the convergence of the solutions of coupled systems to
a common trajectory via static linear feedback under the condition
that the network topology is fixed \cite{tuna08,tuna09}. Later,
time-varying topologies were allowed in \cite{yang11}. As the
limitations of the static feedback have gradually been overcome,
more general results employing dynamic feedback emerged; see, for
instance, \cite{seo09,li10} for fixed and \cite{seo12,li13} for
time-varying topologies.

All of the above-mentioned works, in fact the majority of the
studies on synchronization of dynamical systems, cover the simple
case
\begin{eqnarray}\label{eqn:CTsimple}
{\dot x}_{i}=\sum_{j=1}^{q}a_{ij}(x_{j}-x_{i})\,,\qquad
i=1,\,2,\,\ldots,\,q
\end{eqnarray}
(where $a_{ij}\in\Real_{\geq 0}$ and $x_{i}\in\Real^{n}$) as a
corollary of their main result. An equivalent representation of
these systems reads ${\dot x}=-[L_{1}\otimes I_{n}]x$ where
$x=[x_{1}^{T}\ x_{2}^{T}\ \cdots\ x_{q}^{T}]^{T}$ and
$L_{1}\in\Real^{q\times q}$ is the (weighted) Laplacian matrix
\cite{olfati04} whose spectral properties have proved extremely
useful in the analysis and design of multi-agent systems.

A pleasant thing about \eqref{eqn:CTsimple} is that its geometric
meaning is clear: {\em ``Each agent moves towards the weighted
average of the states of its neighbors."} as stated in
\cite{cao13}. In fact, in the Euler discretization
\begin{eqnarray}\label{eqn:DTsimple}
x_{i}^{+}=x_{i}+\varepsilon\sum_{j=1}^{q}a_{ij}(x_{j}-x_{i})=\sum_{j=1}^{q}w_{ij}x_{j}
\end{eqnarray}
the righthand side becomes {\em the} weighted average for
$\varepsilon>0$ small enough. There are many ways to define {\em
average} and, qualitatively speaking, what any average attempts to
achieve is to compute some sort of {\em center} of the points
considered in the computation. Therefore an excusable and
sometimes even useful choice for weighted arithmetic mean is
obtained by replacing the scalar weights $w_{ij}$ in
\eqref{eqn:DTsimple} by symmetric positive semidefinite matrices
$P_{ij}=P_{ij}^{T}\geq 0$ satisfying $\sum_{j}P_{ij}=I_{n}$. This
suggests on \eqref{eqn:CTsimple} the modification
\begin{eqnarray*}\label{eqn:CTnotsosimple}
{\dot x}_{i}=\sum_{j=1}^{q}Q_{ij}(x_{j}-x_{i})
\end{eqnarray*}
where $Q_{ij}\in\Real^{n\times n}$ are symmetric positive
semidefinite matrices replacing the scalar weights $a_{ij}$. (We
take $Q_{ii}=0$.) Whence follows the dynamics ${\dot x}=-Lx$ where
\begin{eqnarray}\label{eqn:genLap}
L=\left[\begin{array}{cccc}\sum_{j} Q_{1j}&-Q_{12}&\cdots&-Q_{1q}\\
-Q_{21}&\sum_{j}Q_{2j}&\cdots&-Q_{2q}\\
\vdots&\vdots&\ddots&\vdots\\
-Q_{q1}&-Q_{q2}&\cdots&\sum_{j}Q_{qj}
\end{array}\right]_{qn\times qn}
\end{eqnarray}
is the {\em matrix-weighted Laplacian}. In graph theoretical terms
one can say that the graph (with $q$ vertices) associated to this
$L$ is such that to each edge a nonzero positive semidefinite
matrix $Q_{ij}$ is assigned. Note that for the standard Laplacian
the associated graph's edges are assigned weights $a_{ij}$ that
are merely positive scalars.

This paper deals with linear time-invariant systems. We consider a
synchronization problem where the matrix-weighted Laplacian
naturally appears as a tool for both analysis and design. In
particular, we study a group of systems whose uncoupled dynamics
(described by the matrix $A$) are identical and the communication
between each pair $(i,\,j)$ of systems has to be realized via a
(possibly) different output matrix $C_{ij}$. Our goal for this
setup is to generate linear gains $G_{ij}$ to couple the pairs so
that all the solutions in the group converge to a common
trajectory. For $A$ neutrally stable, we achieve this goal under
detectability (of the pairs $(C_{ij},\,A)$ for $C_{ij}\neq 0$) and
symmetry ($C_{ij}=C_{ji}$). We also touch the more general
situation (where $A$ is allowed to yield unbounded solutions) and
establish synchronization under some additional conditions
concerning detectability and the strength of connectivity of the
network topology. We cover both continuous- and discrete-time
cases. Synchronization in an array where each pair of systems are
connected through a different output matrix $C_{ij}$ giving rise
to the matrix-weighted Laplacian is yet a relatively unexplored
area. Among the few works investigating this generalized Laplacian
matrix (in a system-theoretic setting) are \cite{barooah06,
barooah08}, where the authors analyze its spectral properties and
study certain relevant applications in distributed control and
estimation.

\section{Motivation}\label{sec:motivation}

In this section we provide two example arrays of coupled identical
systems where the matrix-weighted Laplacian $L$ appears naturally,
describing the interconnection of individual systems. The first
array is mechanical, the latter electrical.

\subsection{Coupled mass-spring systems}

\begin{figure}[h]
\begin{center}
\includegraphics[scale=0.5]{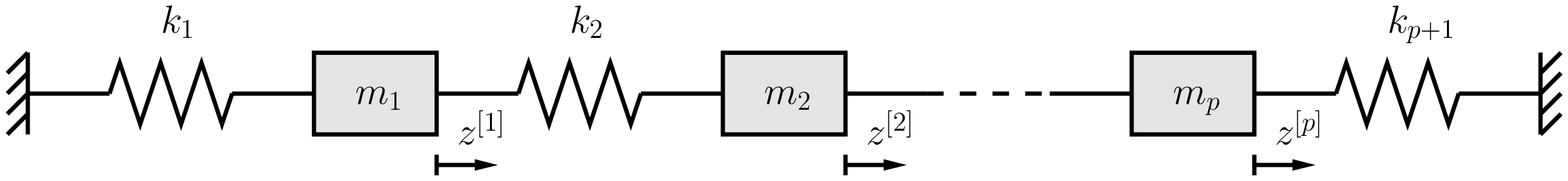}\label{fig:mass-spring}
\caption{Mass-spring system.}
\end{center}
\end{figure}

Consider the individual system in Fig.~1, where $p$ masses are
connected by linear springs. Let $z^{[i]}\in\Real$ be the
displacement of the mass $m_{i}>0$ from the equilibrium. The
spring constants are denoted by $k_{i}>0$. Letting $z=[z^{[1]}\
z^{[2]}\ \cdots\ z^{[p]}]^{T}$ the model of this system reads
$M{\ddot z}+Kz=0$ where $M={\rm
diag}(m_{1},\,m_{2},\,\ldots,\,m_{p})$ and
\begin{eqnarray*}
K=\left[\begin{array}{ccccc}k_{1}+k_{2}&-k_{2}&0&\cdots&0\\-k_{2}&k_{2}+k_{3}&-k_{3}&\cdots&0\\0&-k_{3}&k_{3}+k_{4}&\cdots&0\\
\vdots&\vdots&\vdots&\ddots&\vdots\\0&0&0&\cdots&k_{p}+k_{p+1}\end{array}\right]
\end{eqnarray*}
Let now an array be formed by coupling $q$ replicas of this system
in the arrangement shown in Fig.~2. If we let $z_{i}\in\Real^{p}$
denote the displacement vector for the $i$th system and
$b_{ij}^{[k]}=b_{ji}^{[k]}\geq 0$ represent the viscous friction
(damping) between the $k$th masses of the systems $i$ and $j$, we
can write the dynamics of the coupled systems as $M{\ddot
z}_{i}+Kz_{i}+\sum_{j=1}^{q}B_{ij}({\dot z}_{i}-{\dot z}_{j})=0$
where $B_{ij}={\rm
diag}(b_{ij}^{[1]},\,b_{ij}^{[2]},\,\ldots,\,b_{ij}^{[p]})$.
Letting $x_{i}=[z_{i}^{T}\ {\dot z}_{i}^{T}]^{T}$ denote the state
of the $i$th system we at once obtain
\begin{eqnarray}\label{eqn:mass-spring}
{\dot
x}_{i}=\left[\begin{array}{cc}0&I_{p}\\-M^{-1}K&0\end{array}\right]x_{i}
+\sum_{j=1}^{q}\left[\begin{array}{cc}0&0\\0&M^{-1}B_{ij}\end{array}\right](x_{j}-x_{i})
\end{eqnarray}
Under the coordinate change below
\begin{eqnarray*}
\xi_{i}:=\left[\begin{array}{cc}K^{1/2}&0\\0&M^{1/2}\end{array}\right]x_{i}
\end{eqnarray*}
we can transform \eqref{eqn:mass-spring} into
\begin{eqnarray}\label{eqn:mass-spring-skew}
{\dot \xi}_{i}=S\xi_{i} +\sum_{j=1}^{q}Q_{ij}(\xi_{j}-\xi_{i})
\end{eqnarray}
where
\begin{eqnarray*}
S:=\left[\begin{array}{cc}0&K^{1/2}M^{-1/2}\\-M^{-1/2}K^{1/2}&0\end{array}\right]\quad\mbox{and}\quad
Q_{ij}:=\left[\begin{array}{cc}0&0\\0&M^{-1/2}B_{ij}M^{-1/2}\end{array}\right]
\end{eqnarray*}
Note that $S$ is skew-symmetric and $Q_{ji}=Q_{ij}=Q_{ij}^{T}\geq
0$. Finally, stacking the individual states into a single vector
$\xi=[\xi_{1}^{T}\ \xi_{2}^{T}\ \cdots\ \xi_{q}^{T}]^{T}$ the
dynamics of the array take the form
\begin{eqnarray*}\label{eqn:mass-spring-array}
{\dot \xi}=([I_{q}\otimes S]-L)\xi
\end{eqnarray*}
where $L$ is as defined in \eqref{eqn:genLap}.

\begin{figure}[h]
\begin{center}
\includegraphics[scale=0.45]{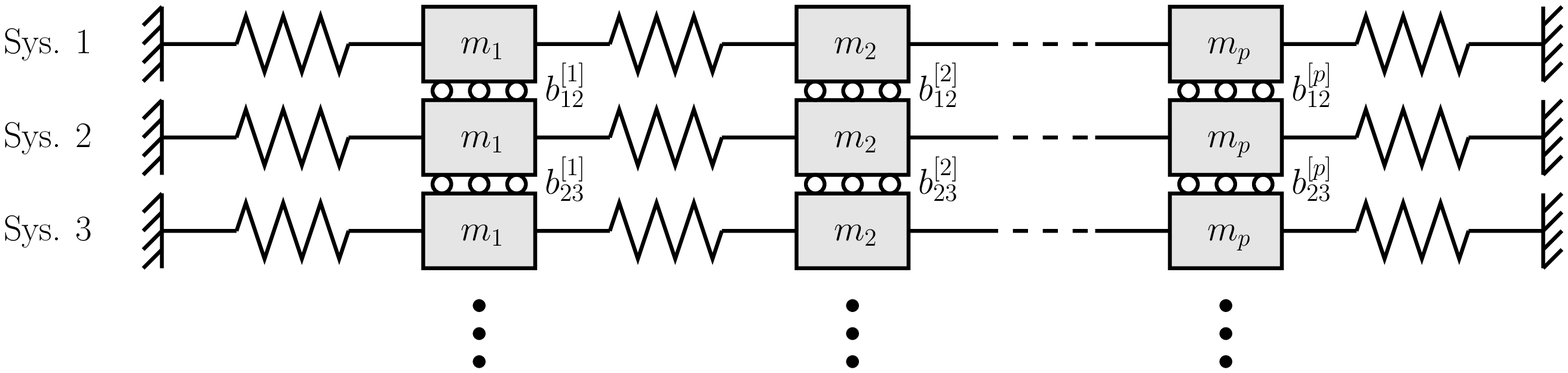}\label{fig:mass-spring-array}
\caption{Array of coupled mass-spring systems.}
\end{center}
\end{figure}

\subsection{Coupled LC oscillators}

\begin{figure}[h]
\begin{center}
\includegraphics[scale=0.5]{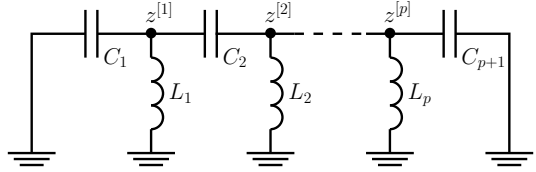}
\caption{LC oscillator system.}
\end{center}
\end{figure}\label{fig:LCcircuit}

Consider the individual system in Fig.~3, where $p$ linear
inductors ($L_{i}>0$) are connected by linear capacitors
($C_{i}>0$). The node voltages are denoted by $z^{[i]}\in\Real$.
Letting $z=[z^{[1]}\ z^{[2]}\ \cdots\ z^{[p]}]^{T}$ the model of
this system reads $C{\ddot z}+L^{-1}z=0$ where $L={\rm
diag}(L_{1},\,L_{2},\,\ldots,\,L_{p})$ and
\begin{eqnarray*}
C=\left[\begin{array}{ccccc}C_{1}+C_{2}&-C_{2}&0&\cdots&0\\-C_{2}&C_{2}+C_{3}&-C_{3}&\cdots&0\\0&-C_{3}&C_{3}+C_{4}&\cdots&0\\
\vdots&\vdots&\vdots&\ddots&\vdots\\0&0&0&\cdots&C_{p}+C_{p+1}\end{array}\right]
\end{eqnarray*}
This time we form the array by coupling $q$ replicas of this
system in the arrangement shown in Fig.~4. If we let
$z_{i}\in\Real^{p}$ denote the node voltage vector for the $i$th
system and $g_{ij}^{[k]}=g_{ji}^{[k]}\geq 0$ be the conductance of
the resistor connecting the $k$th nodes of the systems $i$ and
$j$, we can write the dynamics of the coupled systems as $C{\ddot
z}_{i}+L^{-1}z_{i}+\sum_{j=1}^{q}G_{ij}({\dot z}_{i}-{\dot
z}_{j})=0$ where $G_{ij}={\rm
diag}(g_{ij}^{[1]},\,g_{ij}^{[2]},\,\ldots,\,g_{ij}^{[p]})$.
Letting $x_{i}=[z_{i}^{T}\ {\dot z}_{i}^{T}]^{T}$ denote the state
of the $i$th system we at once obtain
\begin{eqnarray}\label{eqn:LC-oscillator}
{\dot
x}_{i}=\left[\begin{array}{cc}0&I_{p}\\-C^{-1}L^{-1}&0\end{array}\right]x_{i}
+\sum_{j=1}^{q}\left[\begin{array}{cc}0&0\\0&C^{-1}G_{ij}\end{array}\right](x_{j}-x_{i})
\end{eqnarray}
Under the coordinate change below
\begin{eqnarray*}
\xi_{i}:=\left[\begin{array}{cc}L^{-1/2}&0\\0&C^{1/2}\end{array}\right]x_{i}
\end{eqnarray*}
we can transform \eqref{eqn:LC-oscillator} into
\begin{eqnarray}\label{eqn:LC-oscillator-skew}
{\dot \xi}_{i}=S\xi_{i} +\sum_{j=1}^{q}Q_{ij}(\xi_{j}-\xi_{i})
\end{eqnarray}
where
\begin{eqnarray*}
S:=\left[\begin{array}{cc}0&L^{-1/2}C^{-1/2}\\-C^{-1/2}L^{-1/2}&0\end{array}\right]\quad\mbox{and}\quad
Q_{ij}:=\left[\begin{array}{cc}0&0\\0&C^{-1/2}G_{ij}C^{-1/2}\end{array}\right]
\end{eqnarray*}
Note that $S$ is skew-symmetric and $Q_{ji}=Q_{ij}=Q_{ij}^{T}\geq
0$. Finally, as was the case with the mechanical array, the
dynamics of the electrical array reads ${\dot \xi}=([I_{q}\otimes
S]-L)\xi$ where $L$ is the matrix-weighted
Laplacian~\eqref{eqn:genLap}.

\begin{figure}[h]
\begin{center}
\includegraphics[scale=0.5]{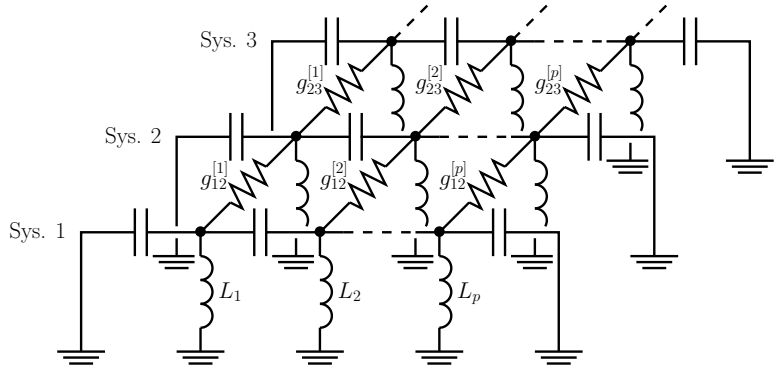}
\caption{Array of LC oscillator systems.}
\end{center}
\end{figure}\label{fig:LCarray}

\section{Problem definition}

In this paper we consider a group of linear systems
\begin{subeqnarray}\label{eqn:CTsystem}
{\dot x}_{i}&=&Ax_{i}+u_{i}\,,\qquad i=1,\,2,\,\ldots,\,q\\
\Y_{i}&=&\{C_{i1}(x_{1}-x_{i}),\,C_{i2}(x_{2}-x_{i}),\,\ldots,\,C_{iq}(x_{q}-x_{i})\}
\end{subeqnarray}
with $A\in\Real^{n\times n}$, where $x_{i}\in\Real^{n}$ is the
state and $u_{i}\in\Real^{n}$ is the (control) input of the $i$th
system. The output set $\Y_{i}$ contains the relative measurements
available to the $i$th system, where $C_{ij}\in\Real^{m_{ij}\times
n}$ and $C_{ii}=0$. Associated to the set $\{C_{ij}\}$, we let the
graph $\G=(\V,\,\setE)$ represent the network topology, where
$\V=\{v_{1},\,v_{2},\,\ldots,\,v_{q}\}$ is the set of vertices and
a pair $(v_{j},\,v_{i})$ belongs to the set of edges $\setE$ when
$C_{ij}\neq 0$.

The problem we study is the stabilization of the synchronization
subspace of the systems~\eqref{eqn:CTsystem}. In particular, we
search for a simple method for choosing the gains
$G_{ij}\in\Real^{n\times m_{ij}}$ such that under the controls
\begin{eqnarray}\label{eqn:CTcontrol}
u_{i}=\sum_{j=1}^{q}G_{ij}C_{ij}(x_{j}-x_{i})
\end{eqnarray}
the systems~\eqref{eqn:CTsystem} (asymptotically) synchronize.
That is, the solutions satisfy $\|x_{i}(t)-x_{j}(t)\|\to 0$ as
$t\to\infty$ for all indices $i,\,j$ and all initial conditions.
We establish synchronization under two different sets of
conditions. We first study the general case where the uncoupled
dynamics ${\dot z}=Az$ are allowed to have unbounded solutions and
provide certain sufficient conditions for synchronization. Later
we will show that if $A$ is neutrally stable, which was the case
with the mechanical and electrical arrays considered earlier, then
synchronization can be achieved under much weaker assumptions.

\section{Synchronization under CL-detectability}\label{sec:UD}

In this section we study synchronization under the assumption
below.

\begin{assumption}\label{assume:CTunstable}
The following conditions hold on the systems~\eqref{eqn:CTsystem}.
\begin{enumerate}
\item $C_{ij}=C_{ji}$ for all $i,\,j$. \item $\G$ is
connected.\footnote{Note that $\G$ becomes undirected under the
first condition.} \item There exists a symmetric positive definite
matrix $P\in\Real^{n\times n}$ such that
\begin{eqnarray}\label{eqn:commonP}
A^{T}P+PA<C_{ij}^{T}C_{ij}\quad\mbox{for all}\quad C_{ij}\neq 0\,.
\end{eqnarray}
\end{enumerate}
\end{assumption}

\begin{remark}
Detectability of a pair $(C_{ij},\,A)$ is equivalent to the
existence of a symmetric positive definite matrix
$P_{ij}\in\Real^{n\times n}$ satisfying
$A^{T}P_{ij}+P_{ij}A<C_{ij}^{T}C_{ij}$. (This is sometimes called
the Lyapunov test for detectability \cite{hespanha09}.) The third
condition of Assumption~\ref{assume:CTunstable} therefore imposes
a certain kind of uniformity on the detectability of the
systems~\eqref{eqn:CTsystem} by letting the detectability of all
the individual pairs $(C_{ij},\,A)$ with $C_{ij}\neq 0$ be
established by a common $P_{ij}=P$. Therefore, referring to the
condition~\eqref{eqn:commonP}, we will henceforth use the term
{\em CL-detectability}, where {\em C} stands for {\em common} and
{\em L} for {\em Lyapunov}.
\end{remark}

Before we state our first theorem we introduce some notation
related to the graph $\G=(\V,\,\setE)$ associated to the
systems~\eqref{eqn:CTsystem}. The degree $d_{i}$ of the vertex
$v_{i}$ is the number of edges satisfying
$(v_{j},\,v_{i})\in\setE$. We let
$\Gamma=[\gamma_{ij}]\in\Real^{q\times q}$ denote the unweighted,
normalized Laplacian matrix defined as follows.
\begin{eqnarray*}
\gamma_{ij}=\left\{\begin{array}{rl}-1/q\,,&(v_{j},\,v_{i})\in\setE\\
d_{i}/q\,,&j=i\\
0\,,&\mbox{elsewhere}\end{array}\right.
\end{eqnarray*}
When $\G$ is undirected and connected, $\Gamma$ is symmetric
positive semidefinite with eigenvalues
$0=\lambda_{1}<\lambda_{2}\leq\cdots\leq\lambda_{q}\leq 1$. In
that case we write $\lambda_{2}(\Gamma)$ to denote the smallest
nonzero eigenvalue of $\Gamma$. Let $\one\in\Real^{q}$ be the
vector of all ones and define $J:=I_{q}-q^{-1}\one\one^{T}$. Note
that $J$ is the unweighted, normalized Laplacian matrix of a
complete graph and satisfies $\lambda_{2}(J)=1$ thanks to $J^2=J$.
Since $J$ and $\Gamma$ share the same eigenvectors (for $\G$
undirected and connected), one can readily establish the bounds
\begin{eqnarray}\label{eqn:Jbound}
\Gamma\leq J\leq \lambda_{2}(\Gamma)^{-1}\Gamma\,.
\end{eqnarray}

\begin{theorem}\label{thm:CTunstable}
Consider the systems~\eqref{eqn:CTsystem} under
Assumption~\ref{assume:CTunstable}. Let $\alpha\geq (2q)^{-1}$ and
$G_{ij}:=\alpha P^{-1}C_{ij}^{T}$ where $P$ satisfies
\eqref{eqn:commonP}. Then under the controls~\eqref{eqn:CTcontrol}
the systems synchronize if
\begin{eqnarray}\label{eqn:epsilonsigma}
\varepsilon>(\lambda_{2}(\Gamma)^{-1}-1)\sigma
\end{eqnarray}
where $\displaystyle \varepsilon:=\min_{C_{ij}\neq
0}\,\lambda_{\rm min}(C_{ij}^{T}C_{ij}-A^{T}P-PA)$ and
$\sigma:=\lambda_{\rm max}(A^{T}P+PA)$.
\end{theorem}

\begin{proof}
Under the suggested controls, dynamics of the
systems~\eqref{eqn:CTsystem} become
\begin{eqnarray}\label{eqn:CTunstableclosedloop}
{\dot
x}_{i}=Ax_{i}+\sum_{j=1}^{q}{\alpha}P^{-1}C_{ij}^{T}C_{ij}(x_{j}-x_{i})\,,\qquad
i=1,\,2,\,\ldots,\,q\,.
\end{eqnarray}
Letting $x=[x_{1}^{T}\ x_{2}^{T}\ \cdots\ x_{q}^{T}]^{T}$ and
$Q_{ij}:=C_{ij}^{T}C_{ij}$ we can rewrite
\eqref{eqn:CTunstableclosedloop} as
\begin{eqnarray}\label{eqn:CTunstablearray}
{\dot x}=\left([I_{q}\otimes A]-\alpha[I_{q}\otimes
P^{-1}]L\right)x=:\Psi x
\end{eqnarray}
where $L$ is the matrix-weighted Laplacian~\eqref{eqn:genLap}.
Since $C_{ij}=C_{ji}$ the matrix $L$ is symmetric. It is also
positive semidefinite because we can write
\begin{eqnarray*}\label{eqn:becauseof}
x^{T}Lx=\sum_{j>i}(x_{j}-x_{i})^{T}Q_{ij}(x_{j}-x_{i})=\sum_{j>i}\|C_{ij}(x_{j}-x_{i})\|^{2}\,.
\end{eqnarray*}
Similarly, we can also write
\begin{eqnarray}\label{eqn:zakoncin}
\lefteqn{x^{T}\left([\Gamma\otimes(A^{T}P+PA)]-q^{-1}L\right)x}\hspace{1in}\nonumber\\
&=&
q^{-1}\sum_{j>i}(x_{j}-x_{i})^{T}(A^{T}P+PA-Q_{ij})(x_{j}-x_{i})\nonumber\\
&\leq&-q^{-1}\sum_{j>i}\varepsilon\|x_{j}-x_{i}\|^{2}\nonumber\\
&=&-\varepsilon x^{T}[\Gamma\otimes I_{n}]x\,.
\end{eqnarray}
By construction $L[\one\otimes I_{n}]=0$. This allows us to write
\begin{eqnarray}\label{eqn:Lcommute}
L[J\otimes I_{n}]&=&L[(I_{q}-q^{-1}\one\one^{T})\otimes I_{n}]\nonumber\\
&=&L[I_{q}\otimes I_{n}]-q^{-1}L[\one\otimes
I_{n}][\one^{T}\otimes I_{n}]\nonumber\\
&=&L\,.
\end{eqnarray}
By symmetry we also have $[J\otimes I_{n}]L=L$. Define
$V:\Real^{qn}\to\Real$ as $V(x):=x^{T}[J\otimes P]x$. We will
employ $V$ as a Lyapunov function for the synchronization subspace
$\{x:x_{i}=x_{j}\ \mbox{for all}\ i,\,j\}\subset\Real^{qn}$. To
this end, let us now study the time derivative of $V$ along the
solutions of the array~\eqref{eqn:CTunstablearray}. Using
\eqref{eqn:Jbound}, \eqref{eqn:zakoncin}, and \eqref{eqn:Lcommute}
we can write
\begin{eqnarray*}
\Psi^{T}[J\otimes P]+[J\otimes P]\Psi &=& [J\otimes
(A^{T}P+PA)]-\alpha(L[J\otimes I_{n}]+[J\otimes I_{n}]L)\\
&=& [J\otimes
(A^{T}P+PA)]-2\alpha L\\
&\leq&[J\otimes (A^{T}P+PA)]-q^{-1}L\\
&=&[(J-\Gamma)\otimes (A^{T}P+PA)]\\
&&\qquad+[\Gamma\otimes
(A^{T}P+PA)]-q^{-1}L\\
&\leq&\sigma[(J-\Gamma)\otimes I_{n}]-\varepsilon[\Gamma\otimes I_{n}]\\
&\leq&(\lambda_{2}(\Gamma)^{-1}-1)\sigma[\Gamma\otimes I_{n}]-\varepsilon[\Gamma\otimes I_{n}]\\
&=&-(\varepsilon-(\lambda_{2}(\Gamma)^{-1}-1)\sigma)[\Gamma\otimes
I_{n}]\,.
\end{eqnarray*}
Therefore we have established
\begin{eqnarray}\label{eqn:Vdot}
\frac{d}{dt}V(x(t))\leq-\delta\, x(t)^{T}[\Gamma\otimes I_{n}]x(t)
\end{eqnarray}
where $\delta:=\varepsilon-(\lambda_{2}(\Gamma)^{-1}-1)\sigma$.
Now, since $[\Gamma\otimes I_{n}]$ is positive semidefinite,
\eqref{eqn:Vdot} implies the following. If
\eqref{eqn:epsilonsigma} holds, i.e., $\delta>0$, then the
solutions converge to the set $\{x:x^{T}[\Gamma\otimes
I_{n}]x=0\}$, which is no other than the synchronization subspace.
\end{proof}
\\

When the graph $\G$ is complete, $\Gamma$ equals $J$ and
$\lambda_{2}(\Gamma)=1$. Then the
condition~\eqref{eqn:epsilonsigma} is satisfied automatically
thanks to \eqref{eqn:commonP}. Hence the result below.

\begin{corollary}\label{cor:complete}
Consider the systems~\eqref{eqn:CTsystem} under
Assumption~\ref{assume:CTunstable}. Let $G_{ij}:=\alpha
P^{-1}C_{ij}^{T}$ where $P$ satisfies \eqref{eqn:commonP} and
$\alpha>0$. Then under the controls~\eqref{eqn:CTcontrol} the
systems synchronize for $\alpha$ large enough and $\G$ complete.
\end{corollary}

Note that any $\alpha\geq(2q)^{-1}$ is large enough by
Theorem~\ref{thm:CTunstable}.


\section{Synchronization under neutral stability}\label{sec:NS}

In the previous section we established synchronization of the
systems~\eqref{eqn:CTsystem} under the CL-detectability
condition~\eqref{eqn:commonP}. In this section we show that this
condition can be relaxed if the uncoupled dynamics harbor only
bounded solutions. To this end, we make the assumption below.

\begin{assumption}\label{assume:CTstable}
The following conditions hold on the systems~\eqref{eqn:CTsystem}.
\begin{enumerate}
\item $C_{ij}=C_{ji}$ for all $i,\,j$. \item $\G$ is connected.
\item $A$ is neutrally stable.\footnote{In the continuous-time
sense. That is, $A$ has no eigenvalue on the open right half-plane
and for each eigenvalue on the imaginary axis the corresponding
Jordan block is one-by-one.} \item The pair $(C_{ij},\,A)$ is
detectable\footnote{In the continuous-time sense. That is, no
eigenvector of $A$ with eigenvalue on the closed right half-plane
belongs to the null space of $C_{ij}$.} for all $C_{ij}\neq 0$.
\end{enumerate}
\end{assumption}

In Section~\ref{sec:motivation} we observed that (after an
appropriate coordinate change) both the mass-spring
systems~\eqref{eqn:mass-spring-skew} and the LC
oscillators~\eqref{eqn:LC-oscillator-skew} were represented by the
highly-structured array dynamics ${\dot \xi}=([I_{q}\otimes
S]-L)\xi$, where $S$ was skew-symmetric and the matrix-weighted
Laplacian $L$ was symmetric. This special righthand side, though
it seems to pertain only to a narrow class of phenomena, in fact
yields readily to generalization. For this reason we study it in
the next lemma, which we will later extend to the main theorem of
this section.

\begin{lemma}\label{lem:CTstable}
Consider the group of systems
\begin{eqnarray}\label{eqn:CTnominal}
{\dot
\xi}_{i}&=&S\xi_{i}+\sum_{j=1}^{q}H_{ij}^{T}H_{ij}(\xi_{j}-\xi_{i})\,,\qquad
i=1,\,2,\,\ldots,\,q
\end{eqnarray}
where $S\in\Real^{n\times n}$ and $H_{ij}\in\Real^{m_{ij}\times
n}$ with $H_{ii}=0$. Let $\setH$ be the graph associated to the
set $\{H_{ij}\}$. Assume that the following hold on the pair
$(S,\,\{H_{ij}\})$.
\begin{itemize}
\item[(C1)] $H_{ij}=H_{ji}$ for all $i,\,j$.  \item[(C2)] $\setH$
is connected. \item[(C3)] $S$ is skew-symmetric. \item[(C4)] The
pair $(H_{ij},\,S)$ is observable for all $H_{ij}\neq 0$.
\end{itemize}
Then the systems synchronize. Moreover, the solutions $\xi_{i}(t)$
remain bounded.
\end{lemma}

\begin{proof}
Letting $\xi=[\xi_{1}^{T}\ \xi_{2}^{T}\ \cdots\ \xi_{q}^{T}]^{T}$
and $Q_{ij}:=H_{ij}^{T}H_{ij}$ we can rewrite
\eqref{eqn:CTnominal} as
\begin{eqnarray*}
\dot\xi=([I_{q}\otimes S]-L)\xi
\end{eqnarray*}
where $L$ is the matrix-weighted Laplacian~\eqref{eqn:genLap}.
Since $H_{ij}=H_{ji}$ the matrix $L$ is symmetric. It is also
positive semidefinite because we can write
\begin{eqnarray*}
x^{T}Lx=\sum_{j>i}(x_{j}-x_{i})^{T}Q_{ij}(x_{j}-x_{i})=\sum_{j>i}\|H_{ij}(x_{j}-x_{i})\|^{2}\,.
\end{eqnarray*}
Thanks to the skew-symmetry of $S$ we have
\begin{eqnarray*}
([I_{q}\otimes S]-L)^{T}+([I_{q}\otimes S]-L)=[I_{q}\otimes
(S+S^{T})]-(L+L^{T})=-2L\,.
\end{eqnarray*}
Thus for the Lyapunov function
$V(\xi)=2^{-1}\xi^{T}\xi=2^{-1}\|\xi\|^{2}$ we can write
\begin{eqnarray*}
\frac{d}{dt}V(\xi(t))&=&-\xi(t)^{T}L\xi(t)\,.
\end{eqnarray*}
Since $L$ is positive semidefinite the solution $\xi(t)$ has to be
bounded. (Hence the boundedness of the solutions $\xi_{i}(t)$.) In
particular, by LaSalle's invariance principle, $\xi(t)$ should
converge to the largest invariant set within the the intersection
$\{\xi:\xi^{T}\xi\leq
\|\xi(0)\|^{2}\}\cap\{\xi:\xi^{T}L\xi=0\}=:\M\subset\Real^{qn}$.
To complete the proof therefore it should suffice to show that in
this largest invariant set we have $\xi_{i}=\xi_{j}$ for all
$i$,\,$j$.

Now let $\xi(t)$ be a solution that belongs identically to $\M$.
Suppose there exist indices $i_{1},\,i_{p}$ such that
\begin{eqnarray}\label{eqn:CTcontradiction}
\xi_{i_{1}}(t)\neq\xi_{i_{p}}(t)
\end{eqnarray}
for some $t\geq 0$. Since $\xi(t)$ belongs identically to $\M$ we
have $L\xi(t)\equiv 0$. In other words
\begin{eqnarray}\label{eqn:CTobs1}
H_{ij}(\xi_{j}(t)-\xi_{i}(t))\equiv 0
\end{eqnarray}
for all $i,\,j$. Then \eqref{eqn:CTnominal} is reduced to
\begin{eqnarray}\label{eqn:CTobs2}
{\dot\xi}_{i}=S\xi_{i}(t)
\end{eqnarray}
for all $i$. By \eqref{eqn:CTobs1}, \eqref{eqn:CTobs2}, and the
observability of the pairs $(H_{ij},\,S)$ (for $H_{ij}\neq 0$) we
can therefore write $\xi_{i}(t)\equiv\xi_{j}(t)$ for all
$H_{ij}\neq 0$. Since the graph $\setH$ is connected we can find
indices $i_{2},\,i_{3},\,\ldots,\,i_{p-1}$ such that
$H_{i_{\ell}i_{\ell+1}}\neq 0$ for $\ell=1,\,2,\,\ldots,\,p-1$.
Then we have $\xi_{i_{\ell}}(t)\equiv\xi_{i_{\ell+1}}(t)$ for
$\ell=1,\,2,\,\ldots,\,p-1$, which implies
$\xi_{i_{1}}(t)\equiv\xi_{i_{p}}(t)$. This contradicts
\eqref{eqn:CTcontradiction}.
\end{proof}
\\

A pleasant pair of byproducts of Lemma~\ref{lem:CTstable} are the
following twin corollaries on the mechanical and electrical arrays
studied in Section~\ref{sec:motivation}.

\begin{corollary}
Consider the coupled mass-spring systems~\eqref{eqn:mass-spring}.
Let the graph associated to the set $\{B_{ij}\}$ be connected and
the pairs $(M^{-1}B_{ij},\,M^{-1}K)$ be observable for all
$B_{ij}\neq 0$. Then the systems synchronize.
\end{corollary}

\begin{corollary}
Consider the coupled LC oscillators~\eqref{eqn:LC-oscillator}. Let
the graph associated to the set $\{G_{ij}\}$ be connected and the
pairs $(C^{-1}G_{ij},\,C^{-1}L^{-1})$ be observable for all
$G_{ij}\neq 0$. Then the oscillators synchronize.
\end{corollary}

In order to extend Lemma~\ref{lem:CTstable} to a general result we
will need the following fact. (Most readers shall find the
statement obvious. Still, for the sake of completeness, a
demonstration is provided.)

\begin{lemma}\label{lem:CTnominal}
Let $A\in\Real^{n\times n}$ be neutrally stable and the signal
$w:\Real_{\geq 0}\to\Real^{n}$ satisfy $\|w(t)\|\leq ce^{-\alpha
t}$ for some constants $c,\,\alpha>0$. Then for each solution
$x(t)$ of the system ${\dot x}(t)=Ax(t)+w(t)$ there exists
$v\in\Real^{n}$ such that $\|x(t)-e^{At}v\|\to 0$ as $t\to\infty$.
\end{lemma}

\begin{proof}
If $A$ is stable (i.e., all its eigenvalues are in the open left
half-plane) then we can choose $v=0$. Otherwise let
$T\in\Real^{n\times n}$ be a transformation matrix such that
\begin{eqnarray*}
T^{-1}AT=\left[\begin{array}{cc}S&0\\0&F\end{array}\right]=:{\tilde
A}
\end{eqnarray*}
where $S\in\Real^{n_{1}\times n_{1}}$ is skew-symmetric and
$F\in\Real^{n_{2}\times n_{2}}$ stable. Apply the coordinate
change $z=[z_{1}^{T}\ z_{2}^{T}]^{T}=T^{-1}x$ with
$z_{1}\in\Real^{n_{1}}$ and $z_{2}\in\Real^{n_{2}}$ and let
$[{\tilde w}_{1}(t)^{T}\ {\tilde w}_{2}(t)^{T}]^{T}=T^{-1}w(t)$
with ${\tilde w}_{1}(t)\in\Real^{n_{1}}$ and ${\tilde
w}_{2}(t)\in\Real^{n_{2}}$. Then we can write
\begin{eqnarray*}
{\dot z}_{1}(t)&=&Sz_{1}(t)+{\tilde w}_{1}(t)\\
{\dot z}_{2}(t)&=&Fz_{2}(t)+{\tilde w}_{2}(t)
\end{eqnarray*}
which yield
\begin{eqnarray*}
z_{1}(t)&=&e^{St}z_{1}(0)+\int_{0}^{t}e^{S(t-\tau)}{\tilde w}_{1}(\tau)d\tau\\
z_{2}(t)&=&e^{Ft}z_{2}(0)+\int_{0}^{t}e^{F(t-\tau)}{\tilde
w}_{2}(\tau)d\tau\,.
\end{eqnarray*}
Note that
\begin{eqnarray}\label{eqn:z2decay}
\lim_{t\to\infty}z_{2}(t)=0
\end{eqnarray}
because $F$ is stable and ${\tilde w}_{2}(t)$ is exponentially
decaying. Let
\begin{eqnarray*}
a:=\int_{0}^{\infty}e^{-S\tau}{\tilde w}_{1}(\tau)d\tau
\end{eqnarray*}
which is well defined because $e^{St}$ is orthogonal (therefore
bounded) and ${\tilde w}_{1}(t)$ is exponentially decaying. In
particular, we have
\begin{eqnarray}\label{eqn:intdecay}
\lim_{t\to\infty}\int_{t}^{\infty}e^{S(t-\tau)}{\tilde
w}_{1}(\tau)d\tau=0\,.
\end{eqnarray}
Finally the below choice
\begin{eqnarray*}
v:=T\left[\begin{array}{c}z_{1}(0)+a\\0\end{array}\right]
\end{eqnarray*}
should work. To see that we write
\begin{eqnarray*}
\|x(t)-e^{At}v\|&=&\|Tz(t)-Te^{{\tilde
A}t}T^{-1}v\|\\
&\leq&\|T\|\left(\|z_{1}(t)-e^{St}(z_{1}(0)+a)\|^{2}+\|z_{2}(t)\|^{2}\right)^{1/2}\\
&=&\|T\|\left(\left\|\int_{t}^{\infty}e^{S(t-\tau)}{\tilde
w}_{1}(\tau)d\tau\right\|^{2}+\|z_{2}(t)\|^{2}\right)^{1/2}\,.
\end{eqnarray*}
The result follows by \eqref{eqn:z2decay} and
\eqref{eqn:intdecay}.
\end{proof}
\\

The below algorithm is where we construct the gains $G_{ij}$ that
ensure synchronization under Assumption~\ref{assume:CTstable}. The
algorithm is followed by the main result of this section.

\begin{algorithm}\label{alg:CTone}
Given $A\in\Real^{n\times n}$ that is neutrally stable and the set
of matrices $\{C_{ij}\}$ with $C_{ij}\in\Real^{m_{ij}\times n}$,
obtain the set $\{G_{ij}\}$ with $G_{ij}\in\Real^{n\times m_{ij}}$
as follows. Let $n_{1}\leq n$ be the number of eigenvalues of $A$
on the imaginary axis and $n_{2}:=n-n_{1}$. If $n_{1}=0$ let
$G_{ij}:=0$. Otherwise, first choose $U\in\Real^{n\times n_{1}}$
and $W\in\Real^{n\times n_{2}}$ satisfying
\begin{eqnarray*}
[U\ \ W]^{-1}A[U\ \
W]=\left[\begin{array}{cc}S&0\\0&F\end{array}\right]
\end{eqnarray*}
where $S\in\Real^{n_{1}\times n_{1}}$ is skew-symmetric and
$F\in\Real^{n_{2}\times n_{2}}$ stable. Then let
$G_{ij}:=UU^{T}C_{ij}^{T}$.
\end{algorithm}

\begin{theorem}\label{thm:CTstable}
Consider the systems~\eqref{eqn:CTsystem} under
Assumption~\ref{assume:CTstable}. Let the gains $G_{ij}$ be
constructed according to Algorithm~\ref{alg:CTone}. Then under the
controls~\eqref{eqn:CTcontrol} the systems synchronize. Moreover,
the solutions $x_{i}(t)$ remain bounded.
\end{theorem}

\begin{proof}
For $n_{1}=0$ the matrix $A$ is stable and the result follows
trivially. Let us hence consider the $n_{1}\geq 1$ case. Under the
suggested controls, dynamics of the systems~\eqref{eqn:CTsystem}
become
\begin{eqnarray}\label{eqn:CTclosedloop}
{\dot
x}_{i}=Ax_{i}+\sum_{j=1}^{q}UU^{T}C_{ij}^{T}C_{ij}(x_{j}-x_{i})\,,\qquad
i=1,\,2,\,\ldots,\,q\,.
\end{eqnarray}
The fist step of the proof is to mold \eqref{eqn:CTclosedloop}
into something we have already studied. To this end, let
$U^{\dagger}\in\Real^{n_{1}\times n}$ and
$W^{\dagger}\in\Real^{n_{2}\times n}$ be such that
\begin{eqnarray*}
\left[\begin{array}{c}U^{\dagger}\\
W^{\dagger}\end{array}\right]=[U\ \ W]^{-1}\,.
\end{eqnarray*}
Then define $\xi_{i}\in\Real^{n_{1}}$ and
$\eta_{i}\in\Real^{n_{2}}$ through the following change of
coordinates
\begin{eqnarray*}
\left[\begin{array}{c}\xi_{i}\\
\eta_{i}\end{array}\right]=\left[\begin{array}{c}U^{\dagger}\\
W^{\dagger}\end{array}\right]x_{i}\,.
\end{eqnarray*}
Now, by letting $H_{ij}:=C_{ij}U$, we can transform
\eqref{eqn:CTclosedloop} into
\begin{subeqnarray}\label{eqn:CTgomlek}
{\dot
\xi}_{i}&=&S\xi_{i}+\sum_{j=1}^{q}H_{ij}^{T}H_{ij}(\xi_{j}-\xi_{i})+\sum_{j=1}^{q}H_{ij}^{T}C_{ij}W(\eta_{j}-\eta_{i})\\
{\dot\eta}_{i}&=&F\eta_{i}
\end{subeqnarray}
thanks to the identities $U^{\dagger}U=I_{n_{1}}$ and
$W^{\dagger}U=0$. The first step is complete.

In the second step we show that the following {\em nominal
systems}
\begin{eqnarray}\label{eqn:nominalCT}
{\dot \xi}_{i}^{\rm nom}&=&S\xi_{i}^{\rm
nom}+\sum_{j=1}^{q}H_{ij}^{T}H_{ij}(\xi_{j}^{\rm nom}-\xi_{i}^{\rm
nom})
\end{eqnarray}
synchronize. This we can do by Lemma~\ref{lem:CTstable} provided
that we show that the conditions (C1)-(C4) are satisfied by the
pair $(S,\{H_{ij}\})$. We have (C1) because $C_{ij}=C_{ji}$. We
have (C3) because $S$ is skew-symmetric by
Algorithm~\ref{alg:CTone}. Let $\setH$ be the graph associated to
the set $\{H_{ij}\}$. Since $\G$ is connected, the equality
$\setH=\G$ would imply (C2). And to show $\setH=\G$ it is enough
that we establish $H_{ij}\neq 0 \Longleftrightarrow C_{ij}\neq 0$.
To this end, we make two simple observations. First,
$(H_{ij},\,S)$ is observable when $(C_{ij},\,A)$ is detectable.
Second, the observability of $(H_{ij},\,S)$ demands $H_{ij}\neq
0$. These observations, in the light of the fact that the pair
$(C_{ij},\,A)$ is detectable for all $C_{ij}\neq 0$, allow us to
construct the following chain of implications.
\begin{eqnarray*}
\begin{array}{ccc}
C_{ij}\neq 0 &\implies& (C_{ij},\,A)\ \mbox{detectable}\\
\Uparrow & \null & \Downarrow\\
H_{ij}\neq 0 & \Longleftarrow & (H_{ij},\,S)\ \mbox{observable}
\end{array}
\end{eqnarray*}
This chain gives us not only the equivalence $H_{ij}\neq 0
\Longleftrightarrow C_{ij}\neq 0$ but also the condition (C4).
This completes the second step.

We begin the last step by stacking the states $\xi=[\xi_{1}^{T}\
\xi_{2}^{T}\ \cdots\ \xi_{q}^{T}]^{T}$, $\eta=[\eta_{1}^{T}\
\eta_{2}^{T}\ \cdots\ \eta_{q}^{T}]^{T}$ and rewriting
\eqref{eqn:CTgomlek} as
\begin{eqnarray}\label{eqn:CTxisystem}
\left[\begin{array}{c}\dot\xi\\\dot\eta\end{array}\right]
=\left[\begin{array}{cc}[I_{q}\otimes S]-L&D\\0&[I_{q}\otimes
F]\end{array}\right]\left[\begin{array}{c}\xi\\\eta\end{array}\right]
\end{eqnarray}
where the structure of $D\in\Real^{qn_{1}\times qn_{2}}$ plays no
role in our analysis and $L$ is the matrix-weighted
Laplacian~\eqref{eqn:genLap} with $Q_{ij}=H_{ij}^{T}H_{ij}$. By
Lemma~\ref{lem:CTstable} the solutions of the
systems~\eqref{eqn:nominalCT} are bounded. This implies that the
block $[I_{q}\otimes S]-L$ has to be neutrally stable. Also, since
$F$ is stable by Algorithm~\ref{alg:CTone}, the block
$[I_{q}\otimes F]$ is stable. Hence the block triangular system
matrix in \eqref{eqn:CTxisystem} is neutrally stable, guaranteeing
the boundedness of the solutions of the
systems~\eqref{eqn:CTgomlek}. Consequently, the solutions
$x_{i}(t)$ of the systems~\eqref{eqn:CTclosedloop} remain bounded.
To show that all $x_{i}(t)$ converge to a common trajectory we
once again look at the system~\eqref{eqn:CTxisystem}. The solution
$\eta(t)$ and, in particular, the term $D\eta(t)$ decay
exponentially because $[I_{q}\otimes F]$ is stable. Since
$[I_{q}\otimes S]-L$ is neutrally stable,
Lemma~\ref{lem:CTnominal} applies to the dynamics
$\dot\xi(t)=([I_{q}\otimes S]-L)\xi(t)+D\eta(t)$ and allows us to
assert that there exists some $v\in\Real^{qn_{1}}$ such that
$\|\xi(t)-e^{([I_{q}\otimes S]-L)t}v\|\to 0$ as $t\to\infty$. In
other words the solutions $\xi_{i}(t)$ converge to the solutions
$\xi^{\rm nom}_{i}(t)$ of the nominal
systems~\eqref{eqn:nominalCT} with $\xi^{\rm nom}(0)=v$, i.e.,
$\|\xi_{i}(t)-\xi_{i}^{\rm nom}(t)\|\to 0$. We know by
Lemma~\ref{lem:CTstable} that the nominal solutions $\xi^{\rm
nom}_{i}(t)$ converge to a common trajectory. This allows us to
claim for the actual solutions that $\|\xi_{i}(t)-\xi_{j}(t)\|\to
0$ for all $i,\,j$. We also have $\|\eta_{i}(t)-\eta_{j}(t)\|\to
0$ since $\eta(t)\to 0$. The synchronization of the
systems~\eqref{eqn:CTclosedloop} then follows because
\begin{eqnarray*}
\|x_{i}(t)-x_{j}(t)\|&=&\|U(\xi_{i}(t)-\xi_{j}(t))+W(\eta_{i}(t)-\eta_{j}(t))\|\\
&\leq&\|U\|\cdot\|\xi_{i}(t)-\xi_{j}(t)\|+\|W\|\cdot\|\eta_{i}(t)-\eta_{j}(t)\|\,.
\end{eqnarray*}
Hence the result.
\end{proof}

\section{Discrete-time problem}

In the last two sections we have established synchronization in an
array of coupled linear systems in continuous time under different
sets of conditions. In Section~\ref{sec:UD} the key assumption for
synchronization was the CL-detectability~\eqref{eqn:commonP} and
in Section~\ref{sec:NS} it was the neutral stability of the
uncoupled dynamics. Now we ask the following question. Can
synchronization be established in discrete time under analogous
assumptions? Our answer is only partial: neutral stability
(through appropriate coupling) does indeed yield synchronization
in discrete time. As for synchronization under
CL-detectability\footnote{In the discrete-time sense. That is,
there exists a common symmetric positive definite matrix $P$
satisfying $A^{T}PA-P<C_{ij}^{T}C_{ij}$ for all $C_{ij}\neq 0$.}
all our attempts to generate the discrete-time counterpart of
Theorem~\ref{thm:CTunstable} have so far proved futile.

In this section we study the discrete-time version of the problem
that was attended to in Section~\ref{sec:NS}. The road map we
adopt is parallel to that of the continuous-time case, causing at
times some pardonable repetitions. Consider the group of
discrete-time linear systems
\begin{subeqnarray}\label{eqn:DTsystem}
x_{i}^{+}&=&Ax_{i}+u_{i}\,,\qquad i=1,\,2,\,\ldots,\,q\\
\Y_{i}&=&\{C_{i1}(x_{1}-x_{i}),\,C_{i2}(x_{2}-x_{i}),\,\ldots,\,C_{iq}(x_{q}-x_{i})\}
\end{subeqnarray}
where $x_{i}^{+}$ denotes the state of the $i$th system at the
next time instant, $A\in\Real^{n\times n}$, and
$C_{ij}\in\Real^{m_{ij}\times n}$ with $C_{ii}=0$. As before, we
let the graph $\G$ (associated to the set $\{C_{ij}\}$) represent
the network topology. Here, analogous to the continuous-time
problem, we search for a simple method for choosing the gains
$G_{ij}\in\Real^{n\times m_{ij}}$ such that under the controls
\begin{eqnarray}\label{eqn:DTcontrol}
u_{i}=\varepsilon\sum_{j=1}^{q}G_{ij}C_{ij}(x_{j}-x_{i})
\end{eqnarray}
(for $\varepsilon>0$ sufficiently small) the
systems~\eqref{eqn:DTsystem} synchronize. That is, the solutions
satisfy $\|x_{i}(k)-x_{j}(k)\|\to 0$ as $k\to\infty$
($k\in\Natural$) for all indices $i,\,j$ and all initial
conditions. We make the following assumption.

\begin{assumption}\label{assume:DTstable}
The following conditions hold on the systems~\eqref{eqn:DTsystem}.
\begin{enumerate}
\item $C_{ij}=C_{ji}$ for all $i,\,j$. \item $\G$ is connected.
\item $A$ is neutrally stable.\footnote{In the discrete-time
sense. That is, $A$ has no eigenvalue with magnitude larger than
one and for each eigenvalue on the unit circle the corresponding
Jordan block is one-by-one.} \item The pair $(C_{ij},\,A)$ is
detectable\footnote{In the discrete-time sense. That is, no
eigenvector of $A$ with eigenvalue on or outside the unit circle
belongs to the null space of $C_{ij}$.} for all $C_{ij}\neq 0$.
\end{enumerate}
\end{assumption}

As in continuous-time case, we first analyze a simpler problem
(Lemma~\ref{lem:DTstable}) which inspires a method to generate the
coupling gains $G_{ij}$. This method is then elaborated in
Algorithm~\ref{alg:DTone} and why it should work is demonstrated
in our discrete-time main result Theorem~\ref{thm:DTstable}.

\begin{lemma}\label{lem:DTstable}
Consider the group of systems
\begin{eqnarray}\label{eqn:DTnominal}
\xi^{+}_{i}&=&Q\xi_{i}+\varepsilon
Q\sum_{j=1}^{q}H_{ij}^{T}H_{ij}(\xi_{j}-\xi_{i})\,,\qquad
i=1,\,2,\,\ldots,\,q
\end{eqnarray}
where $Q\in\Real^{n\times n}$ and $H_{ij}\in\Real^{m_{ij}\times
n}$ with $H_{ii}=0$. Let $\setH$ be the graph associated to the
set $\{H_{ij}\}$. Assume that the following hold on the pair
$(Q,\,\{H_{ij}\})$.
\begin{itemize}
\item[(D1)] $H_{ij}=H_{ji}$ for all $i,\,j$.  \item[(D2)] $\setH$
is connected. \item[(D3)] $Q$ is orthogonal. \item[(D4)] The pair
$(H_{ij},\,Q)$ is observable for all $H_{ij}\neq 0$.
\end{itemize}
Let $L$ be the matrix-weighted Laplacian~\eqref{eqn:genLap} with
$Q_{ij}:=H_{ij}^{T}H_{ij}$ and $\bar\varepsilon>0$ satisfy
$L\geq\bar\varepsilon L^{2}$. Then for all
$\varepsilon\in(0,\,\bar\varepsilon]$ the systems synchronize
under the controls~\eqref{eqn:DTcontrol}. Moreover, the solutions
$\xi_{i}(k)$ remain bounded.
\end{lemma}

\begin{proof}
Suppose $\varepsilon\in(0,\,\bar\varepsilon]$. Then we have
$2L-\varepsilon L^{2}\geq L$. Letting $\xi=[\xi_{1}^{T}\
\xi_{2}^{T}\ \cdots\ \xi_{q}^{T}]^{T}$ we can rewrite
\eqref{eqn:DTnominal} as
\begin{eqnarray*}
\xi^{+}=[I_{q}\otimes Q](I_{nq}-\varepsilon L)\xi\,.
\end{eqnarray*}
Since $H_{ij}=H_{ji}$ the matrix $L$ is symmetric. It is also
positive semidefinite (see the proof of Lemma~\ref{lem:CTstable}).
Since $Q$ is orthogonal we have
\begin{eqnarray*}
\lefteqn{([I_{q}\otimes Q](I_{nq}-\varepsilon L))^{T}[I_{q}\otimes
Q](I_{nq}-\varepsilon L)-I_{nq}}\\&&=(I_{nq}-\varepsilon
L)[I_{q}\otimes (Q^{T}Q)](I_{nq}-\varepsilon L)-I_{nq}\\
&&=(I_{nq}-\varepsilon L)^{2}-I_{nq}\\
&&=-2\varepsilon L+\varepsilon^{2}L^{2}\,.
\end{eqnarray*}
Therefore employing the Lyapunov function
$V(\xi)=\xi^{T}\xi=\|\xi\|^{2}$ we can write
\begin{eqnarray*}
V(\xi(k+1))-V(\xi(k))&=&-\varepsilon\xi^{T}(k)(2L-\varepsilon
L^{2})\xi(k)\\
&\leq&-\varepsilon\xi^{T}(k)L\xi(k)\,.
\end{eqnarray*}
Since $L$ is positive semidefinite the solution $\xi(k)$ has to be
bounded. (Hence the boundedness of the solutions $\xi_{i}(k)$.) In
particular, by LaSalle's invariance principle, $\xi(k)$ should
converge to the largest invariant set within the the intersection
$\{\xi:\xi^{T}\xi\leq
\|\xi(0)\|^{2}\}\cap\{\xi:\xi^{T}L\xi=0\}=:\M\subset\Real^{qn}$.
Using the same simple arguments employed in the proof of
Lemma~\ref{lem:CTstable}, one can show that in this largest
invariant set we have $\xi_{i}=\xi_{j}$ for all $i$,\,$j$.
\end{proof}
\\

The following fact is the discrete-time version of
Lemma~\ref{lem:CTnominal}. It will find use in the proof of the
discrete-time main result.

\begin{lemma}\label{lem:DTnominal}
Let $A\in\Real^{n\times n}$ be neutrally stable and the signal
$w:\Natural\to\Real^{n}$ satisfy $\|w(k)\|\leq ce^{-\alpha k}$ for
some constants $c,\,\alpha>0$. Then for each solution $x(k)$ of
the system $x(k+1)=Ax(k)+w(k)$ there exists $v\in\Real^{n}$ such
that $\|x(k)-A^{k}v\|\to 0$ as $k\to\infty$.
\end{lemma}

The below algorithm is where we construct the gains $G_{ij}$ that
ensure synchronization under Assumption~\ref{assume:DTstable}. The
statement following the algorithm is the discrete-time counterpart
of Theorem~\ref{thm:CTstable}.

\begin{algorithm}\label{alg:DTone}
Given $A\in\Real^{n\times n}$ that is neutrally stable and the set
of matrices $\{C_{ij}\}$ with $C_{ij}\in\Real^{m_{ij}\times n}$,
obtain the set $\{G_{ij}\}$ with $G_{ij}\in\Real^{n\times m_{ij}}$
as follows. Let $n_{1}\leq n$ be the number of eigenvalues of $A$
on the unit circle and $n_{2}:=n-n_{1}$. If $n_{1}=0$ let
$G_{ij}:=0$. Otherwise, first choose $U\in\Real^{n\times n_{1}}$
and $W\in\Real^{n\times n_{2}}$ satisfying
\begin{eqnarray*}
[U\ \ W]^{-1}A[U\ \
W]=\left[\begin{array}{cc}Q&0\\0&F\end{array}\right]
\end{eqnarray*}
where $Q\in\Real^{n_{1}\times n_{1}}$ is orthogonal and
$F\in\Real^{n_{2}\times n_{2}}$ stable\footnote{In the
discrete-time sense. That is, all the eigenvalues of $A$ are on
the open unit disk.}. Then let $G_{ij}:=UQU^{T}C_{ij}^{T}$.
\end{algorithm}

\begin{theorem}\label{thm:DTstable}
Consider the systems~\eqref{eqn:DTsystem} under
Assumption~\ref{assume:DTstable}. Let the gains $G_{ij}$ be
constructed according to Algorithm~\ref{alg:DTone}. Also let $L$
be the matrix-weighted Laplacian~\eqref{eqn:genLap} with
$Q_{ij}:=U^{T}C_{ij}^{T}C_{ij}U$ and $\bar\varepsilon>0$ satisfy
$L\geq\bar\varepsilon L^{2}$. Then for all
$\varepsilon\in(0,\,\bar\varepsilon]$ the systems synchronize
under the controls~\eqref{eqn:DTcontrol}. Moreover, the solutions
$x_{i}(k)$ remain bounded.
\end{theorem}

\begin{proof}
For $n_{1}=0$ the matrix $A$ is stable and the result follows
trivially. Let us hence consider the $n_{1}\geq 1$ case. Under the
suggested controls, dynamics of the systems~\eqref{eqn:DTsystem}
become
\begin{eqnarray}\label{eqn:DTclosedloop}
x_{i}^{+}=Ax_{i}+\varepsilon\sum_{j=1}^{q}UQU^{T}C_{ij}^{T}C_{ij}(x_{j}-x_{i})\,,\qquad
i=1,\,2,\,\ldots,\,q\,.
\end{eqnarray}
The fist step of the proof is to mold \eqref{eqn:DTclosedloop}
into something we have already studied. To this end, let
$U^{\dagger}\in\Real^{n_{1}\times n}$ and
$W^{\dagger}\in\Real^{n_{2}\times n}$ be such that
\begin{eqnarray*}
\left[\begin{array}{c}U^{\dagger}\\
W^{\dagger}\end{array}\right]=[U\ \ W]^{-1}\,.
\end{eqnarray*}
Then define $\xi_{i}\in\Real^{n_{1}}$ and
$\eta_{i}\in\Real^{n_{2}}$ through the following change of
coordinates
\begin{eqnarray*}
\left[\begin{array}{c}\xi_{i}\\
\eta_{i}\end{array}\right]=\left[\begin{array}{c}U^{\dagger}\\
W^{\dagger}\end{array}\right]x_{i}\,.
\end{eqnarray*}
Now, by letting $H_{ij}:=C_{ij}U$, we can transform
\eqref{eqn:DTclosedloop} into
\begin{subeqnarray}\label{eqn:DTgomlek}
\xi_{i}^{+}&=&Q\xi_{i}+\varepsilon Q\sum_{j=1}^{q}H_{ij}^{T}H_{ij}(\xi_{j}-\xi_{i})+\varepsilon Q\sum_{j=1}^{q}H_{ij}^{T}C_{ij}W(\eta_{j}-\eta_{i})\\
\eta_{i}^{+}&=&F\eta_{i}
\end{subeqnarray}
thanks to the identities $U^{\dagger}U=I_{n_{1}}$ and
$W^{\dagger}U=0$. The first step is complete.

In the second step we claim that the following nominal systems
\begin{eqnarray}\label{eqn:nominalDT}
\xi^{\rm nom}_{i}(k+1)&=&Q\xi_{i}^{\rm nom}(k)+\varepsilon
Q\sum_{j=1}^{q}H_{ij}^{T}H_{ij}(\xi_{j}^{\rm nom}(k)-\xi_{i}^{\rm
nom}(k))
\end{eqnarray}
synchronize. The claim follows from Lemma~\ref{lem:DTstable} once
the conditions (D1)-(D4) are shown to be satisfied by the pair
$(Q,\{H_{ij}\})$. This we can achieve by emulating the part the
proof of Theorem~\ref{thm:CTstable} where the conditions (C1)-(C4)
were shown to hold for the systems~\eqref{eqn:CTnominal}.

We begin the last step by stacking the states $\xi=[\xi_{1}^{T}\
\xi_{2}^{T}\ \cdots\ \xi_{q}^{T}]^{T}$ and $\eta=[\eta_{1}^{T}\
\eta_{2}^{T}\ \cdots\ \eta_{q}^{T}]^{T}$. Then
\eqref{eqn:DTgomlek} can be rewritten as
\begin{eqnarray}\label{eqn:DTxisystem}
\left[\begin{array}{c}\xi^{+}\\\eta^{+}\end{array}\right]
=\left[\begin{array}{cc}[I_{q}\otimes Q](I_{nq}-\varepsilon
L)&D\\0&[I_{q}\otimes
F]\end{array}\right]\left[\begin{array}{c}\xi\\\eta\end{array}\right]
\end{eqnarray}
for some $D\in\Real^{qn_{1}\times qn_{2}}$. By
Lemma~\ref{lem:DTstable} the solutions of the
systems~\eqref{eqn:nominalDT} are bounded. This implies that the
block $[I_{q}\otimes Q](I_{nq}-\varepsilon L)$ has to be neutrally
stable. Also, since $F$ is stable by Algorithm~\ref{alg:DTone},
the block $[I_{q}\otimes F]$ is stable. Hence the block triangular
system matrix in \eqref{eqn:DTxisystem} is neutrally stable,
guaranteeing the boundedness of the solutions of the
systems~\eqref{eqn:DTgomlek}. Consequently, the solutions
$x_{i}(k)$ of the systems~\eqref{eqn:DTclosedloop} remain bounded.
To show that all $x_{i}(k)$ converge to a common trajectory we
once again look at the system~\eqref{eqn:DTxisystem}. The solution
$\eta(k)$ and, in particular, the term $D\eta(k)$ decay
exponentially because $[I_{q}\otimes F]$ is stable. Since
$[I_{q}\otimes Q](I_{nq}-\varepsilon L)$ is neutrally stable,
Lemma~\ref{lem:DTnominal} applies to the dynamics
$\xi(k+1)=[I_{q}\otimes Q](I_{nq}-\varepsilon L)\xi(k)+D\eta(k)$
and allows us to assert that there exists some
$v\in\Real^{qn_{1}}$ such that $\|\xi(k)-([I_{q}\otimes
Q](I_{nq}-\varepsilon L))^{k}v\|\to 0$ as $k\to\infty$. In other
words the solutions $\xi_{i}(k)$ converge to the solutions
$\xi^{\rm nom}_{i}(k)$ of the nominal
systems~\eqref{eqn:nominalDT} with $\xi^{\rm nom}(0)=v$, i.e.,
$\|\xi_{i}(k)-\xi_{i}^{\rm nom}(k)\|\to 0$. We know by
Lemma~\ref{lem:DTstable} that the nominal solutions $\xi^{\rm
nom}_{i}(k)$ converge to a common trajectory. This allows us to
claim for the actual solutions that $\|\xi_{i}(k)-\xi_{j}(k)\|\to
0$ for all $i,\,j$. We also have $\|\eta_{i}(k)-\eta_{j}(k)\|\to
0$ since $\eta(k)\to 0$. The synchronization of the
systems~\eqref{eqn:DTclosedloop} then follows.
\end{proof}

\section{Notes}

All the results in the paper rest on the symmetry of the
underlying matrix-weighted Laplacian matrix. In other words, we
only consider the case where the graph (whose edges are assigned
matrix values) representing the network topology is undirected.
Now, for synchronization problems involving a scalar-weighted
Laplacian, the symmetry assumption has long been shed because it
is redundant. This raises the following question. Can we still
guarantee synchronization if we remove the symmetry condition on
the matrix-weighted Laplacian? A more technical, but much easier
to answer version of this question is: Can we remove the condition
(C1) from Lemma~\ref{lem:CTstable}? The answer is no. Below is a
counterexample.

\begin{example}
Consider the following three coupled systems in $\Real^{2}$
\begin{eqnarray}\label{eqn:counterex}
{\dot
x}_{i}=Sx_{i}+\sum_{j=1}^{3}H_{ij}^{T}H_{ij}(x_{j}-x_{i})\,,\qquad
i=1,\,2,\,3
\end{eqnarray}
with $H_{ii}=0$ and
\begin{eqnarray*}
\begin{array}{rclrcl}
H_{12}&=&\left[\begin{array}{rr}1.9006&1.8406\\1.8406&4.0758\end{array}\right]\,,\quad&
H_{13}&=&\left[\begin{array}{rr}1.0382&0.9603\\0.9603&6.2512\end{array}\right]\,,\\&&&&&\\
H_{21}&=&\left[\begin{array}{rr}3.8896&3.1418\\3.1418&4.7041\end{array}\right]\,,\quad&
H_{23}&=&\left[\begin{array}{rr}6.4288&-1.6342\\-1.6342&1.5263\end{array}\right]\,,\\&&&&&\\
H_{31}&=&\left[\begin{array}{rr}2.2944&-1.9328\\-1.9328&6.5011\end{array}\right]\,,\quad&
H_{32}&=&\left[\begin{array}{rr}4.9157&-3.9794\\-3.9794&3.6283\end{array}\right]\,.
\end{array}
\end{eqnarray*}
All these nonzero $H_{ij}$ are nonsingular. Now, the graph $\setH$
associated to the set $\{H_{ij}\}$ is complete because $H_{ij}\neq
0$ for all $i\neq j$. Therefore $\setH$ is connected. Take
$S\in\Real^{2\times 2}$ to be the zero matrix $S=0$. Then $S$ is
trivially skew-symmetric. Also, all the pairs $(H_{ij},\,S)$ are
observable for $H_{ij}\neq 0$ since the nonzero $H_{ij}$ are full
column rank. Hence the systems~\eqref{eqn:counterex} satisfy the
conditions (C2)-(C4) of Lemma~\ref{lem:CTstable}. The only
condition being violated is (C1) because $H_{ij}\neq H_{ji}$.
Without this condition the symmetry of the Laplacian
\begin{eqnarray*}
L=\left[\begin{array}{ccc}
H_{12}^{T}H_{12}+H_{13}^{T}H_{13}&-H_{12}^{T}H_{12}&-H_{13}^{T}H_{13}\\
-H_{21}^{T}H_{21}&H_{21}^{T}H_{21}+H_{23}^{T}H_{23}&-H_{23}^{T}H_{23}\\
-H_{31}^{T}H_{31}&-H_{32}^{T}H_{32}&H_{31}^{T}H_{31}+H_{32}^{T}H_{32}
\end{array}\right]
\end{eqnarray*}
is broken and the synchronization is not achieved for this
example. In particular, the array dynamics $\dot{x}=-Lx$  has an
unstable eigenvalue $\lambda=4.0312$ whose eigenvector does not
belong to the synchronization subspace
$\{x:x_{1}=x_{2}=x_{3}\}\subset\Real^{6}$.
\end{example}

The previous example sheds some light on the symmetry issue by
saying that the undirectedness of the network graph (though it
might still be a conservative constraint) is not altogether
removable when one wants to achieve synchronization under
matrix-weighted Laplacian. Another issue we would like to address,
in order to have a better feel of the degree of necessity of
certain assumptions, is related to the
condition~\eqref{eqn:epsilonsigma} in
Theorem~\ref{thm:CTunstable}. Once $\varepsilon$ and $\sigma$ are
fixed, since $\lambda_{2}(\Gamma)$ is a measure of graph
connectivity, the equation~\eqref{eqn:epsilonsigma} can be
interpreted as: {\em the more connected the network graph the more
likely the synchronization.} In fact, as stated in
Corollary~\ref{cor:complete}, in the limiting case where the graph
is complete, the synchronization is certain under the feedback
gains $G_{ij}=\alpha P^{-1}C_{ij}^{T}$ for large enough coupling
coefficient $\alpha$. One can also show that when all the output
matrices are identical (up to a scaling) $C_{ij}=\rho_{ij} C$
(with $\rho_{ij}=\rho_{ji}$) connectedness of the graph is enough
for synchronization (for large enough $\alpha$). Now it is
impossible not to ask the next question. Can we remove the
completeness assumption from Corollary~\ref{cor:complete}? The
answer is once again negative as shown by the counterexample
below.

\begin{example}
Consider five systems in $\Real^{3}$ with
dynamics~\eqref{eqn:CTsystem}. We take
\begin{eqnarray*}
A=\left[\begin{array}{rrr}
0.4429 & 0.4871 & 0.7504\\
0.7265 & -1.5839 & -1.8779\\
0.0154 & 1.3969 & 1.5767
\end{array}\right]
\end{eqnarray*}
This system matrix $A$ is not stable due to an eigenvalue at
$\lambda=0.9678$. As for the output matrices, the nonzero $C_{ij}$
are
\begin{eqnarray*}
\begin{array}{ccccl}
C_{12} &=& C_{21} &=& [\ 1\ \ \ 0\ \ \ 0\ ]\,,\\\vspace{-0.1in}\\
C_{23} &=& C_{32} &=& [\ 3.3036\ \ \ 0.1565\ \ \ 0.1265\ ]\,,\\\vspace{-0.1in}\\
C_{34} &=& C_{43} &=& [\ 3.7854\ \ \ 1.3147\ \ \ 3.4819\ ]\,,\\\vspace{-0.1in}\\
C_{45} &=& C_{54} &=& [\ 4.6054\ \ \ 1.8354\ \ \ 3.1269\ ]\,.
\end{array}
\end{eqnarray*}
The associated graph $\G$ with five vertices
$\{v_{1},\,v_{2},\,\ldots,\,v_{5}\}$ has a simple chain structure
$v_{1}\longleftrightarrow v_{2}\longleftrightarrow
v_{3}\longleftrightarrow v_{4}\longleftrightarrow v_{5}$. Hence
$\G$ is connected, but not complete. Also, these systems are
CL-detectable~\eqref{eqn:commonP} with the following symmetric
positive definite matrix
\begin{eqnarray*}
P=\left[\begin{array}{rrr}
0.6209 & -0.1396 & -0.2605\\
-0.1396 & 0.0677 & 0.0997\\
-0.2605 & 0.0997 & 0.1815
\end{array}\right]
\end{eqnarray*}
In particular, we have
\begin{eqnarray*}
\lambda_{\rm
max}(A^{T}P+PA-C_{ij}^{T}C_{ij})\leq-0.0047\quad\mbox{for
all}\quad C_{ij}\neq 0\,.
\end{eqnarray*}
Note that all the conditions listed in
Assumption~\ref{assume:CTunstable} are satisfied for our example.
Suppose now that we couple these five systems through the feedback
gains $G_{ij}=\alpha P^{-1}C_{ij}^{T}$ (suggested in
Corollary~\ref{cor:complete}) where we leave the coupling
coefficient $\alpha>0$ as a design parameter. Since the graph $\G$
is not complete, Corollary~\ref{cor:complete} is silent, meaning
that we have to resort to simulation results to determine whether
synchronization can be achieved for large enough $\alpha$.
Consider now the matrix $\Psi$ representing the coupled array
dynamics~\eqref{eqn:CTunstablearray}. In our case,
$\Psi\in\Real^{15\times 15}$ has 15 eigenvalues. Out of these 15
eigenvalues, the three of them, say
$\lambda_{1},\,\lambda_{2},\,\lambda_{3}$, equal the three
eigenvalues of $A$ and their eigenvectors belong to the
synchronization subspace
$\{x:x_{1}=x_{2}=x_{3}=x_{4}=x_{5}\}\subset\Real^{15}$. For
synchronization to take place, it is necessary that all the
remaining eigenvalues
$\lambda_{4},\,\lambda_{5},\,\ldots,\,\lambda_{15}$ are in the
open left half-plane. Fig.~5 displays the variation of
$\rho:=\max_{i\geq 4}{\rm Re}(\lambda_{i})$ with respect to
$\alpha$. Note that $\rho$ never gets negative. In fact it seems
to satisfy $\rho(\alpha)\geq 0.0418$ for all $\alpha$. Hence, for
the example at hand, synchronization cannot be achieved by
adjusting the coupling coefficient.

\begin{figure}[h]
\begin{center}
\includegraphics[scale=0.65]{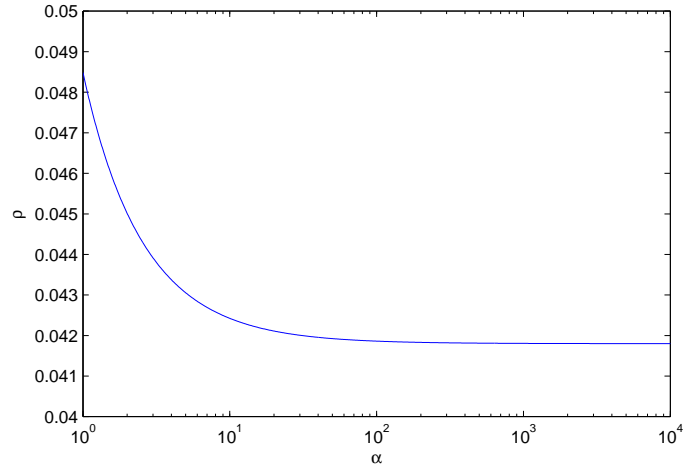}
\caption{Variation of $\rho$ with respect to the coupling
coefficient $\alpha$.}
\end{center}
\end{figure}\label{fig:rho}

\end{example}

\bibliographystyle{plain}
\bibliography{references}
\end{document}